\documentclass[11pt,reqno]{amsart}
\usepackage[T1]{fontenc}
\usepackage[utf8]{inputenc}
\usepackage{lmodern}
\usepackage{amsmath,amssymb,amsthm,mathtools}
\usepackage[margin=1.08in]{geometry}
\usepackage{microtype}
\usepackage{enumitem}
\usepackage{xcolor}

\usepackage[colorlinks=true,linkcolor=blue,citecolor=blue,urlcolor=blue]{hyperref}
\usepackage[capitalise,noabbrev]{cleveref}

\setlist[enumerate]{label=\textup{(\roman*)},leftmargin=*,itemsep=3pt}
\numberwithin{equation}{section}
\newtheorem{theorem}{Theorem}[section]
\newtheorem{proposition}[theorem]{Proposition}
\newtheorem{lemma}[theorem]{Lemma}
\newtheorem{corollary}[theorem]{Corollary}
\theoremstyle{definition}
\newtheorem{definition}[theorem]{Definition}
\theoremstyle{remark}
\newtheorem{remark}[theorem]{Remark}

\newcommand{\kk}{k}
\newcommand{\eps}{\varepsilon}
\newcommand{\id}{\operatorname{id}}
\newcommand{\udim}{\operatorname{udim}}
\newcommand{\rann}{\operatorname{rann}}
\newcommand{\lann}{\operatorname{lann}}
\newcommand{\Mat}{\mathrm{M}}
\newcommand{\Span}{\operatorname{span}}
\newcommand{\op}{\mathrm{op}}
\newcommand{\cop}{\mathrm{cop}}
\allowdisplaybreaks[2]

\title[A noetherian Hopf algebra has bijective antipode]{A noetherian Hopf algebra has bijective antipode}
\author[A.~Sciandra]{Andrea Sciandra}
\address{%
\parbox[b]{\linewidth}{Département de Mathématiques, Université Libre de Bruxelles, Boulevard du Triomphe, B-1050 Bruxelles, Belgium}}
\email{andrea.sciandra@ulb.be}
\urladdr{\url{www.andreasciandra.com}}
\date{}
\hypersetup{
  pdftitle={One-sided uniform dimension and bijectivity of the antipode},
  pdfsubject={Complete candidate argument for the one-sided noetherian antipode conjecture}
}

\begin{document}

\begin{abstract}
We prove that a Hopf algebra over an arbitrary field has a bijective antipode if either its left or right regular module has finite uniform dimension. In particular, this solves Skryabin’s conjecture, posed in 2006, that every noetherian Hopf algebra has a bijective antipode. Starting from Skryabin’s injectivity theorem for weakly finite Hopf algebras, we associate to each finite-dimensional subcoalgebra an essential left ideal of common denominators that is also a right coideal in the antipode image. An annihilator identity makes this image essential in the original algebra as a left module over the image, yielding finite left uniform dimension from finite right uniform dimension. A second denominator construction inside the antipode image, together with a restricted inverse-antipode identity, proves surjectivity.
\end{abstract}
\maketitle

\tableofcontents

\section{Introduction}

The bijectivity of the antipode is a basic structural question for
infinite-dimensional Hopf algebras. In 2006, Skryabin conjectured that
every noetherian Hopf algebra has a bijective antipode
\cite[p.~623]{Sk06}. The conjecture is still listed as an open problem
in Brown's recent survey \cite[Question 2.8 and \S4.1]{Br26}.
Skryabin's injectivity theorem for weakly finite Hopf algebras
\cite[Theorem A (i)]{Sk06} applies, in particular, under noetherianity
on either side. Thus, in the noetherian case, the remaining issue
is surjectivity.

Several special cases of the conjecture were previously established.
Skryabin proved bijectivity for semiprime right noetherian Hopf
algebras and for affine noetherian PI Hopf algebras
\cite[Corollaries 1 and 2]{Sk06}.
Bijectivity is also known for noetherian AS-Gorenstein Hopf algebras
\cite[Corollary~0.3]{LOWY18}, and for residually finite-dimensional
Hopf algebras that are right or left noetherian
\cite[Theorem 5.3]{Sk21}.

In this paper, we prove the stronger criterion in Theorem \ref{thm:main}:
a Hopf algebra over an arbitrary field has a bijective antipode
whenever its left or right regular module has finite uniform
dimension. In particular, every Hopf algebra that is noetherian on either
side has a bijective antipode, resolving Skryabin's conjecture.

The main construction assigns to each finite-dimensional subcoalgebra
an essential left ideal of common denominators in the antipode image,
which is also a right coideal. A coideal annihilator identity then
makes this image essential in the original algebra as a module over
the image, yielding the uniform-dimension bound in Proposition \ref{prop:uniform-bound}. A second denominator construction, now
inside the antipode image, produces a nonzero intersection with the
next image. The inverse-antipode identity on its natural domain forces
the first denominator ideal to contain the unit. Surjectivity follows
by the fundamental theorem of coalgebras.

\section{Preliminaries}

We recall some notions and results that will be useful in the following.
All algebras, coalgebras, vector spaces, and tensor products are over a
fixed arbitrary field \(\kk\). Algebras and algebra morphisms are
unital, and modules are unital. Coalgebras and coalgebra morphisms are counital, and comodules are counital. We use Sweedler notation for calculations involving the coproduct. The antipode is always denoted by \(S\), referring to the
Hopf algebra under consideration. For a Hopf algebra \(H\), a \emph{right coideal} is a
subspace \(L\subseteq H\) satisfying $\Delta(L)\subseteq L\otimes H$. A left coideal satisfies \(\Delta(L)\subseteq H\otimes L\). The counit cannot vanish identically on a nonzero right
or left coideal.

We recall the standard notions of essential submodules and uniform
(or Goldie) dimension; see e.g.\ \cite[\S3D and \S6A]{Lam99}.

\begin{definition}
Let \(M\) be a left \(R\)-module. A submodule \(E\subseteq M\) is \emph{essential} if every nonzero
submodule of \(M\) has nonzero intersection with \(E\).
The \emph{uniform dimension} \(\udim(M)\) is the supremum of the
cardinalities of independent families of nonzero submodules, where
independent means that their sum is direct.
A nonzero module is \emph{uniform} if any two of its nonzero
submodules have nonzero intersection.
\end{definition}

For a left \(R\)-module \(M\), essentiality of \(E\subseteq M\)
is equivalently the following elementwise condition: for every
\(0\neq m\in M\), there is \(r\in R\) such that
\(0\neq rm\in E\). Indeed, apply the definition to the cyclic
submodule \(Rm\); conversely, every nonzero submodule contains such
a cyclic submodule.

We use the elementary module-theoretic properties that uniform dimension
is monotone on submodules and additive on finite direct sums. In
particular, if \(\udim(M)=d<\infty\), then
\(\udim(M^m)=md\). Every noetherian module has finite uniform dimension \cite[Corollary 6.7]{Lam99}.
For completeness, the latter assertion follows by successively choosing
uniform submodules: a nonzero noetherian module has a uniform submodule,
since otherwise repeated splittings into two disjoint nonzero submodules
would produce an infinite direct sum. If a finite sum of the chosen
uniform submodules is not essential, choose another one disjoint from
that sum. Noetherianity forces this process to stop. Otherwise the successive partial direct sums would form a
strictly increasing chain of submodules. The resulting finite
direct sum is essential and has finite uniform dimension.

We sum up some properties that will be needed throughout the paper.

\begin{lemma}[{\cite[Exercise 6 (A) in \S3 and Corollary 6.10 (2)]{Lam99}}]\label{lem:essential-basics}
The following properties hold for modules over any unital ring.
\begin{enumerate}
\item A finite intersection of essential submodules is essential.
\item If \(E\subseteq M\) is essential and \(\udim(E)=d<\infty\),
then \(\udim(M)=d\).
\item If \(\udim(M)<\infty\) and \(f:M\to M\) is an injective
module endomorphism, then \(f(M)\) is essential in \(M\).
\end{enumerate}
\end{lemma}




\begin{definition}
A ring \(R\) is \emph{weakly finite} (also called \emph{stably finite}; see \cite[\S1B]{Lam99}) if, for every \(m\geq1\),
\[
  A,B\in\Mat_m(R),\quad AB=I_m\quad\Longrightarrow\quad BA=I_m.
\]
\end{definition}

\begin{lemma}[{\cite[\S1E, p.~17, and Exercise~1 in \S6]{Lam99}\label{lem:weakly-finite}}]
If \(\udim(R_R)<\infty\) or \(\udim({}_R R)<\infty\), then
\(R\) is weakly finite. Consequently every ring that is noetherian
on either side is weakly finite.
\end{lemma}


\begin{theorem}[Skryabin {\cite[Theorem A (i)]{Sk06}}]
\label{thm:skryabin}
The antipode of a weakly finite Hopf algebra over a field is injective.
\end{theorem}

\begin{corollary}\label{cor:injective}
If a Hopf algebra has finite uniform dimension on either side, its
antipode is injective. This applies, in particular, under noetherianity
on either side.
\end{corollary}

\begin{proof}
Combine Lemma \ref{lem:weakly-finite} and Theorem \ref{thm:skryabin}.
\end{proof}

If \(S\) is the antipode of a Hopf algebra \(R\), its image \(B:=S(R)\)
is a Hopf subalgebra of $R$. 
We recall that \(S(xy)=S(y)S(x)\), \(S(1_{R})=1_{R}\), \(\Delta(S(x))=\sum S(x_{(2)})\otimes S(x_{(1)})\),
and \(\eps(S(x))=\eps(x)\). 
When \(S\) is injective,
\(S:R\to B\) is an anti-isomorphism, and hence
\begin{equation}\label{eq:antipode-lattice}
  \udim({}_B B)=\udim(R_R).
\end{equation}
This equality follows directly by transporting right ideals of \(R\)
to left ideals of \(B\); sums, intersections, and independence are
preserved. We refer the reader to \cite{Mo93,Ra12} for basic notions and results about Hopf algebras.

\section{Finite coefficient spaces and common denominators}

We now prove some technical results that will be useful later. In the next lemma, (i) is the fundamental theorem of coalgebras, see e.g.\ \cite[Theorem 2.2.3]{Ra12}, while (ii) is the coefficient-matrix construction of
\cite[Lemma 2.2.1 and Corollary 2.2.2]{Ra12}.
We include the proof to fix the coefficient conventions used below.

\begin{lemma}
\label{lem:coefficients}
Let \(R\) be a coalgebra. The following statements hold.
\begin{enumerate}
\item Every finite subset of \(R\) lies in a finite-dimensional
subcoalgebra.
\item Every nonzero finite-dimensional subcoalgebra \(C\subseteq R\)
is spanned by the entries of a matrix \(U=(u_{ij})_{ij}\) satisfying
\begin{equation}\label{eq:multiplicative}
  \Delta(u_{ij})=\sum_a u_{ia}\otimes u_{aj},
  \qquad \eps(u_{ij})=\delta_{ij}.
\end{equation}
Such a matrix will be called multiplicative.
\end{enumerate}
\end{lemma}

\begin{proof}
(i). Given \(x\in R\), write \(\Delta(x)=\sum_i a_i\otimes b_i\) with
the \(b_i\) linearly independent. The finite-dimensional space
\[
  V_x=\{(\id\otimes f)\Delta(x):f\in R^*\}
     =\Span\{a_i\}
\]
contains \(x\), by the counit identity. For every \(f\in R^*\), coassociativity gives
\[
  \Delta\bigl((\id\otimes f)\Delta(x)\bigr)
  =\sum_i a_i\otimes(\id\otimes f)\Delta(b_i)
  \in V_x\otimes R,
\]
thus \(\Delta(V_x)\subseteq V_x\otimes R\). A finite sum $V$ of such spaces
is a finite-dimensional right coideal containing any prescribed finite
subset. Choose a basis \(v_1,\ldots,v_m\) of this right coideal \(V\), and
write $\Delta(v_j)=\sum_i v_i\otimes u_{ij}$.
Coassociativity and the right counit identity give
\[
  \sum_{i,a}v_a\otimes u_{ai}\otimes u_{ij}
  =\sum_i v_i\otimes\Delta(u_{ij}),
  \qquad
  v_j=\sum_i v_i\eps(u_{ij}).
\]
Comparing the coefficients of the basis vectors \(v_i\) yields
\eqref{eq:multiplicative}. Thus the span of the \(u_{ij}\) is a
finite-dimensional subcoalgebra. The left counit identity gives $v_j=\sum_i\eps(v_i)u_{ij}$,
so this subcoalgebra contains \(V\).

(ii). Take \(V=C\). All its coefficient entries then belong to
\(C\), and the previous identity shows that they span \(C\).
\end{proof}

For a multiplicative matrix over a Hopf algebra with antipode \(S\),
the two antipode identities read
\begin{equation}\label{eq:matrix-inverse}
  U\,S(U)=S(U)\,U=I_m.
\end{equation}
There is no transposition in this identity. Here \(S(U)=(S(u_{ij}))_{ij}\) is defined entrywise: the
\((i,j)\)-entries of the two products are
\(\sum_a u_{ia}S(u_{aj})\) and
\(\sum_a S(u_{ia})u_{aj}\), respectively. These are precisely the
two antipode identities applied to \(u_{ij}\).

\begin{proposition}\label{prop:denominators}
Let \(R\) be a Hopf algebra with antipode \(S\), and set \(B:=S(R)\).
Assume \(\udim({}_B B)<\infty\). For every finite-dimensional
subcoalgebra \(C\subseteq R\), the space
\begin{equation}\label{eq:denominator-definition}
  \mathcal L_C=\{b\in B:bC\subseteq B\}
\end{equation}
is an essential left ideal of \(B\) and a right coideal of \(B\).
In particular, it contains an element \(b\) with \(\eps(b)=1\).
\end{proposition}

\begin{proof}
If \(C=0\), then \(\mathcal L_C=B\). Otherwise, choose a
multiplicative matrix \(U\) spanning \(C\), as in Lemma
\ref{lem:coefficients}. Set \(A:=S(U)\in\Mat_m(B)\).
Rows \(B^{1\times m}=\{(r_1,\ldots,r_m):r_i\in B\}\) are regarded as a left \(B\)-module. The map
\[
  f:{}_B B^{1\times m}\longrightarrow{}_B B^{1\times m},
  \quad r\longmapsto rA
\]
is left \(B\)-linear. It is injective since \(rA=0\) implies
\(r=rAU=0\) in \(R^{1\times m}\), using \eqref{eq:matrix-inverse}.
Its domain has finite uniform dimension, so $E:=B^{1\times m}A$
is essential by Lemma \ref{lem:essential-basics} (iii).

For each standard row \(e_i\), set
\[
  E_i:=\{b\in B:be_i\in E\}.
\]
It is an essential left ideal of $B$: for a nonzero left ideal \(I\subseteq B\),
the submodule \(Ie_i\) is nonzero and intersects \(E\) nontrivially.
Moreover, $be_i\in E$ if and only if $be_iU\in B^{1\times m}$.
Indeed, \(be_i=rA\) implies \(be_iU=r\); conversely,
\((be_iU)A=be_i\). Since the entries of \(U\) span \(C\), we get
\begin{equation}\label{eq:coordinate-intersection}
  \mathcal L_C=\bigcap_{i=1}^m E_i.
\end{equation}
It follows that \(\mathcal L_C\) is essential by Lemma \ref{lem:essential-basics} (i), and nonzero. It is
also visibly a left ideal, since \((db)C=d(bC)\subseteq B\) for
\(d\in B\) and \(b\in\mathcal L_C\).

For its coideal property, let \(b\in\mathcal L_C\) and \(c\in C\). We have that
\begin{equation}\label{eq:denominator-coideal}
\sum\Delta(bc_{(1)})(1\otimes S(c_{(2)}))=\sum b_{(1)}c_{(1)}\otimes b_{(2)}c_{(2)}S(c_{(3)})=\sum b_{(1)}c\otimes b_{(2)}.
\end{equation}
The left-hand side is in \(B\otimes B\): we have \(bC\subseteq B\),
\(\Delta(B)\subseteq B\otimes B\), and \(S(C)\subseteq B\).

Write $\Delta(b)=\sum_{\ell=1}^s p_\ell\otimes q_\ell$,
with \(p_\ell,q_\ell\in B\) and the \(q_\ell\) linearly independent.
For every \(c\in C\), \eqref{eq:denominator-coideal} gives
\(\sum_\ell p_\ell c\otimes q_\ell\in B\otimes B\). Let \(\pi:R\to R/B\) be the vector-space quotient map, and choose
\(\lambda_\ell\in B^*\) with
\(\lambda_\ell(q_j)=\delta_{\ell j}\). Applying
\(\pi\otimes\lambda_\ell\) to this tensor gives
\(\pi(p_\ell c)=0\), so \(p_\ell c\in B\).
As this holds for every \(c\in C\), each \(p_\ell\) belongs to
\(\mathcal L_C\). Therefore $\Delta(\mathcal L_C)\subseteq\mathcal L_C\otimes B$.

Since $\mathcal L_C$ is a nonzero right coideal, some element of \(\mathcal L_C\) has nonzero counit, and
rescaling gives counit one.
\end{proof}

We call a nonzero element \(b\in\mathcal L_C\)
a common left denominator for \(C\) relative to \(B\).

\begin{remark}\label{rem:denominator-scope}
The finite-uniform-dimension hypothesis is used on the specific left
module \({}_B B^{1\times m}\). No finiteness of \({}_B R\) is assumed. The matrix \(U\) is used in the ambient matrix ring \(\mathrm{M}_m(R)\); its entries are not assumed to belong to \(B\). 
\end{remark}

\section{Coideal annihilators and essentiality of the antipode image}

We combine the common-denominator construction with a coideal
annihilator identity. This makes the antipode image essential in the
ambient Hopf algebra as a left module over that image and transfers
finite uniform dimension from right to left.

\begin{lemma}\label{lem:coideal-annihilator}
Let \(R\) be a Hopf algebra with antipode \(S\).
\begin{enumerate}
\item If \(0\neq L\subseteq R\) is a right coideal, then
\[
  \rann_R(L):=\{h\in R:Lh=0\}\subseteq\ker S.
\]
\item If \(0\neq P\subseteq R\) is a left coideal, then
\[
  \lann_R(P):=\{h\in R:hP=0\}\subseteq\ker S.
\]
\end{enumerate}
In particular, the indicated annihilators vanish when \(S\) is injective.
\end{lemma}

\begin{proof}

(i). We already know that \(\eps|_L\) cannot vanish. 
Choose and rescale \(\ell\in L\) so that \(\eps(\ell)=1\). Write
\[
  \Delta(\ell)=\sum_i a_i\otimes b_i,\qquad a_i\in L.
\]
If \(Lh=0\), then \(a_i h=0\), so
\(S(h)S(a_i)=S(a_i h)=0\). Consequently, we get
\[
  S(h)=S(h)\sum_i S(a_i)b_i=\sum_i[S(h)S(a_i)]b_i=0.
\]
(ii). Choose \(p\in P\) with \(\eps(p)=1\), and write \(\Delta(p)=\sum_i a_i\otimes b_i\) with \(b_i\in P\). If \(hP=0\), then \(S(b_i)S(h)=0\), and hence $S(h)=\sum_i a_i S(b_i)S(h)=0$.
\end{proof}

\begin{proposition}\label{prop:essential-image-extension}
Let \(H\) be a Hopf algebra with injective antipode \(S\), and set \(K:=S(H)\).
Suppose that, for every finite-dimensional subcoalgebra \(C\subseteq H\),
there is a nonzero right coideal \(D_C\) of \(H\) such that
\[
  D_C\subseteq K,\qquad D_CC\subseteq K.
\]
Then \({}_K K\) is an essential submodule of \({}_K H\).
\end{proposition}

\begin{proof}
Let \(X\subseteq H\) be a nonzero left \(K\)-submodule, and choose
\(0\neq h\in X\). By Lemma \ref{lem:coefficients}, there is a
finite-dimensional subcoalgebra \(C\subseteq H\) containing \(h\).
By Lemma \ref{lem:coideal-annihilator} (i) and injectivity of \(S\),
\(\rann_H(D_C)=0\). Therefore \(D_Ch\neq0\), and we may choose
\(a\in D_C\) such that \(ah\neq0\).

The product \(ah\) belongs to \(X\), because \(a\in K\) and
\(X\) is a left \(K\)-submodule. It belongs to \(K\), because
\(h\in C\) and \(D_CC\subseteq K\). Thus
\(0\neq ah\in X\cap K\), proving essentiality.
\end{proof}

\begin{proposition}
\label{prop:uniform-bound}
Let \(H\) be a Hopf algebra with
\(d=\udim(H_H)<\infty\), and set \(K:=S(H)\). Then
\begin{equation}\label{eq:uniform-bound}
  \udim({}_H H)\leq\udim({}_K H)
  =\udim({}_K K)=d.
\end{equation}
\end{proposition}

\begin{proof}
The antipode is injective by Corollary \ref{cor:injective}. The anti-isomorphism
\(S:H\to K\) gives \(\udim({}_K K)=d\) by
\eqref{eq:antipode-lattice}. Thus Proposition \ref{prop:denominators} applies to
\(R=H\) and \(B=K\).
For every finite-dimensional subcoalgebra \(C\subseteq H\), it gives
a nonzero space
\[
  \mathcal L_C=\{a\in K:aC\subseteq K\},
  \qquad \Delta(\mathcal L_C)\subseteq \mathcal L_C\otimes K.
\]
In particular, \(\mathcal L_C\) is a right coideal of \(H\) itself and
satisfies the hypotheses for \(D_C\) in Proposition
\ref{prop:essential-image-extension}. Hence
\({}_K K\subseteq{}_K H\) is essential.
By Lemma \ref{lem:essential-basics} (ii),
\[
  \udim({}_K H)=\udim({}_K K)=d.
\]
Finally, every left ideal of \(H\) is a left \(K\)-submodule of
the same underlying space \(H\), and independence of a family of
such ideals is unchanged by restriction of scalars. This proves the
first inequality in \eqref{eq:uniform-bound}.
\end{proof}

\begin{remark}\label{rem:new-step}
The bound \eqref{eq:uniform-bound} is obtained before applying any
second denominator construction. It does not assume finite left
uniform dimension of \(H\), left noetherianity of \(H\), or
noetherianity of \({}_K H\). The relevant annihilator is the common
annihilator \(\rann_H(\mathcal L_C)\), not the annihilator of an individual
chosen denominator. This is what ensures a nonzero product inside
\(X\cap K\).
\end{remark}

\section{Restricted inverse and the second denominator}

We use an identity for the inverse of the antipode on its image,
together with a second denominator construction. This proves
bijectivity under finite uniform dimension on both sides.

\begin{lemma}\label{lem:restricted-inverse}
Let \(R\) be a Hopf algebra with injective antipode \(S\).
Let \(L\subseteq R\) be a left ideal and a right coideal.
If \(L\cap S(R)\neq0\), then \(L=R\).
\end{lemma}

\begin{proof}
Set \(B:=S(R)\). Over a field $k$, the tensor intersection identity $(L\otimes R)\cap(B\otimes B)=(L\cap B)\otimes B$
holds inside \(R\otimes R\). In fact, restrict the map
\((R\to R/L)\otimes\id\) to \(B\otimes B\); its kernel is
\((L\cap B)\otimes B\). Indeed, the kernel of \(B\to R/L\) is \(L\cap B\),
and tensoring over the field \(\kk\) preserves kernels and the
inclusion \(B\subseteq R\). It follows that $\Delta(L\cap B)\subseteq(L\cap B)\otimes B$. Thus the nonzero intersection $L\cap B$ is a right coideal of \(B\), and we
can choose \(\bar{y}\in L\cap B\) with \(\eps(\bar{y})=1\).

We consider the inverse linear map
\(S^{-1}:B\to R\). For every \(y\in B\), writing \(y=S(x)\), one has
\begin{equation}\label{eq:restricted-inverse-identity}
  \sum S^{-1}(y_{(2)})y_{(1)}=\sum S^{-1}(S(x_{(1)}))S(x_{(2)})
  =\sum x_{(1)}S(x_{(2)})=\eps(x)1=\eps(y)1.
\end{equation}
For the chosen \(\bar{y}\), write
\(\Delta(\bar{y})=\sum_i y_i\otimes z_i\), where
\(y_i\in L\cap B\) and \(z_i\in B\). Thus every \(S^{-1}(z_i)\) is defined and belongs to \(R\). Since \(L\) is
a left ideal, \eqref{eq:restricted-inverse-identity} gives $1=\sum_i S^{-1}(z_i)y_i\in L$, so $L=R$.
\end{proof}

\begin{proposition}
\label{prop:two-stage}
Let \(H\) be a Hopf algebra with antipode $S$ for which both
\(\udim(H_H)\) and \(\udim({}_H H)\) are finite.
Then its antipode is bijective.
\end{proposition}

\begin{proof}
By Corollary \ref{cor:injective}, \(S\) is injective. Define
\[
  K:=S(H),\qquad J:=S(K)=S^2(H).
\]
Both are Hopf subalgebras of $H$. The anti-isomorphism \(S:H\to K\)
and the isomorphism \(S^2:H\to J\) give
\begin{equation}\label{eq:two-left-dimensions}
  \udim({}_K K)=\udim(H_H)<\infty,
  \qquad
  \udim({}_J J)=\udim({}_H H)<\infty.
\end{equation}
Fix a nonzero finite-dimensional subcoalgebra \(C\subseteq H\).
Apply Proposition \ref{prop:denominators} to \(R=H\) and \(B=K\).
It gives an essential left ideal and right coideal of $K$
\[
  \mathcal L_C=\{a\in K:aC\subseteq K\}.
\]
Choose \(a\in \mathcal L_C\) with \(\eps(a)=1\), and choose a
finite-dimensional subcoalgebra \(C'\subseteq K\) containing \(a\).

Apply Proposition \ref{prop:denominators} again, to the Hopf algebra
\(R=K\) which has antipode \(S|_K\), and its image \(B=J\).
The second finiteness assertion in \eqref{eq:two-left-dimensions}
is precisely the hypothesis needed for this application.
It yields an element \(b\in J\) such that
\[
  bC'\subseteq J,\qquad \eps(b)=1.
\]
The product \(ba\) has three distinct properties:
\begin{enumerate}
\item \(ba\in \mathcal L_C\), since \(b\in J\subseteq K\) and \(\mathcal L_C\)
is a left ideal of \(K\);
\item \(ba\in J\), since \(a\in C'\) and \(bC'\subseteq J\);
\item \(ba\neq0\), since \(\eps(ba)=\eps(b)\eps(a)=1\).
\end{enumerate}
Consequently, we have obtained
\[
  0\neq ba\in\mathcal L_C\cap J=\mathcal L_C\cap S(K).
\]
The antipode $S|_K$ of $K$ is injective and has image $J$. We apply Lemma \ref{lem:restricted-inverse} to \(R=K\) and \(L=\mathcal L_C\) (the restricted inverse used in that lemma is therefore \(\left(S|_K\right)^{-1}:J\longrightarrow K\)). Since \(\mathcal L_C\) is a left ideal and a right coideal of \(K\), and $\mathcal{L}_C\cap S(K)\not=0$, the lemma gives \(\mathcal L_C=K\). Hence \(1\in \mathcal L_C\), and therefore \(C\subseteq K\).

Every element of \(H\) is contained in a finite-dimensional
subcoalgebra by Lemma \ref{lem:coefficients}. Therefore \(H=K=S(H)\),
and \(S\) is bijective.
\end{proof}

\section{The one-sided conclusion}

We combine the preceding results to prove bijectivity from finite
uniform dimension on at least one side. In particular, this solves
Skryabin's conjecture, posed in \cite{Sk06}.

\begin{theorem}\label{thm:main}
Let \(H\) be a Hopf algebra over an arbitrary field.
If the regular module of \(H\) has finite uniform dimension on
either side, then its antipode is bijective.
In particular, every right noetherian Hopf algebra and every left
noetherian Hopf algebra has bijective antipode.
\end{theorem}

\begin{proof}
Suppose first that \(\udim(H_H)<\infty\).
By Proposition \ref{prop:uniform-bound}, also \(\udim({}_H H)<\infty\).
The conclusion follows from Proposition \ref{prop:two-stage}.

If \(\udim({}_H H)<\infty\), consider
\(\widetilde H=H^{\op,\cop}\). This is a Hopf algebra with antipode given by the same linear map \(S\). 
The underlying ring is \(H^{\op}\), so we have
\[
  \udim(\widetilde H_{\widetilde H})=\udim({}_H H)<\infty.
\]
The right-sided case applied to \(\widetilde H\) shows that the same
linear map \(S\) is bijective.

Finally, noetherianity on a specified side implies finite uniform
dimension on that side, so both noetherian assertions follow.
\end{proof}

\begin{corollary}\label{cor:symmetry}
Let $H$ be a Hopf algebra over an arbitrary field.
\begin{enumerate}
\item[(i)]
The module ${}_H H$ has finite uniform dimension if and only
if $H_H$ does. When these dimensions are finite, they are equal.

\item[(ii)]
$H$ is right noetherian if and only if it is
left noetherian.
\end{enumerate}
\end{corollary}

\begin{proof}
(i). If either uniform dimension is finite, Theorem \ref{thm:main} shows that $S$ is bijective.
As an anti-automorphism of $H$, it identifies the lattices
of right and left ideals, preserving sums and intersections.
It therefore preserves independence of families of ideals,
giving equality of the two uniform dimensions.

(ii). Noetherianity on either side again implies
bijectivity of $S$ by Theorem \ref{thm:main}. The same correspondence between right and left ideals
preserves ascending chains, proving the assertion.
\end{proof}

\begin{remark}
The first denominator construction uses finite left uniform dimension
of \(K\), which comes from finite right uniform dimension of \(H\)
via \(S:H\to K\). The essentiality argument in Proposition
\ref{prop:essential-image-extension} and Proposition \ref{prop:uniform-bound} then proves
finite left uniform dimension of \(H\). Only after this step is
finite left uniform dimension of \(J\) obtained via \(S^2:H\to J\)
and used in the second denominator construction. 
\end{remark}

\noindent\textbf{Acknowledgments.} The author is a member of the “National Group for Algebraic and Geometric Structures and their Applications” (GNSAGA-INdAM) and is supported by a postdoctoral fellowship at the ULB within the framework of the PDR project “Reconstruction of modules and algebraic objects from closed and monoidal structures on their representation categories” funded by the FNRS under the grant number T.0318.25F (PI J. Vercruysse). 

\noindent\textbf{Use of AI.} The proof strategy was suggested by GPT-6 (OpenAI) through an iterative dialogue with the author; intermediate results helped identify the key difficulties and arrive at a simpler argument. All AI-assisted material used in this paper has been checked by the author, who takes full responsibility for the content of the paper.

\end{document}